\documentclass[utf8,11pt]{article}
\usepackage[utf8]{inputenc}
\usepackage{fullpage}
\usepackage{authblk}
\usepackage{amsthm,amsmath,amssymb}
\usepackage{hyperref}
\usepackage{xcolor}
\hypersetup{
    colorlinks,
    linkcolor={red!50!black},
    citecolor={blue!50!black},
    urlcolor={blue!80!black}
}

\theoremstyle{plain}
\newtheorem{theorem}{Theorem}
\newtheorem{lemma}[theorem]{Lemma}
\newtheorem{corollary}[theorem]{Corollary}
\newtheorem{proposition}[theorem]{Proposition}

\newtheorem{claim}{Claim}[theorem]

\theoremstyle{definition}

\title{A Polynomial Time
	Algorithm for the $k$-Disjoint Shortest Paths Problem}
\author[1]{William Lochet}
\affil[1]{\small University of Bergen, Norway}

\date{}

\begin{document}
\maketitle

\begin{abstract}
    The disjoint paths problem is a fundamental problem in algorithmic graph theory and combinatorial optimization. For a given graph $G$ and a set of $k$ pairs of terminals in $G$, it asks for the existence of $k$ vertex-disjoint paths connecting each pair of terminals. The proof of Robertson and Seymour [JCTB 1995] of the existence of an $n^3$ algorithm for any fixed $k$ is one of the highlights of their Graph Minors project. In this paper, we focus on the version of the problem where all the paths are required to be shortest paths. This problem, called the disjoint shortest paths problem, was introduced by Eilam-Tzoreff [DAM 1998] where she proved that the case $k = 2$ admits a polynomial time algorithm. This problem has received some attention lately, especially since the proof of the existence of a polynomial time algorithm in the directed case when $k = 2$ by Bérczi and Kobayashi [ESA 2017]. However, the existence of a polynomial algorithm when $k = 3$ in the undirected version remained open since 1998. 
    
    In this paper we show that for any fixed $k$, the disjoint shortest paths problem admits a polynomial time algorithm. In fact for any fixed $C$, the algorithm can be extended to treat the case where each path connecting the pair $(s,t)$ has length at most $d(s,t) + C$.  
\end{abstract}

\section{Introduction}

Given a graph $G$ and a set of pairs of vertices $(s_1, t_1), \dots, (s_k, t_k)$, 
the \textit{Vertex-Disjoint Paths Problem} asks whether there exists a set of 
vertex-disjoint paths $P_1, \dots, P_k$ such that every $P_i$ is an $(s_i,t_i)$ path. 
This is a classical NP-hard \cite{karp1975} problem in graph theory with applications to VLSI-layouts and networks problem \cite{schrijver1990paths} which has been extensively studied. 
The proof of the existence of a $O(n^3)$ algorithm for any fixed $k$
by Robertson and Seymour \cite{Seymour95} is one of the highlights from the Graph Minors project, and the running time has later been improved to $O(n^2)$ 
\cite{KAWARABAYASHI2012424}.
In the directed case, the problem is NP-hard even for $k =2$ \cite{FORTUNE1980111}, but some results are known 
for special classes of digraphs like acyclic digraphs \cite{FORTUNE1980111}, planar digraphs 
\cite{schrijver1994finding} or tournaments \cite{CHUDNOVSKY2015582}. 
 
One natural question is, given an instance of the Vertex-Disjoint Path Problem, to find a solution which minimises the sum of the lengths of the $P_i$. This problem appears to be much harder, as only the case $k = 2$ was recently solved by Björklund and Husfeldt \cite{Bjorklund}. In fact, even deciding if the problem admits an optimal solution, i.e where every $P_i$ is a shortest path between $s_i$ and $t_i$ is open for $k \geq 3$. 

This problem was first considered by Eilam-Tzoreff \cite{EILAMTZOREFF1998113} 20 years ago. In the same paper, she gave an algorithm for the case $k=2$ 
and conjectured that a polynomial algorithm exists for any fixed $k$, both in the directed and undirected setting. This problem has received some attention lately. In particular Bérczi and Kobayashi \cite{berczi2017directed} gave a nice proof that the directed version admits a polynomial time algorithm when $k = 2$. They also show that the algorithm of Schrijver \cite{schrijver1994finding} can be used to prove Eilam-Tzoreff's conjecture when the input (di)graph is planar.
The goal of this paper is to solve this problem for any $k$ in the undirected case.

\begin{theorem}\label{thm:SDP}
	For any fixed integer $k$, there exists an algorithm running in $n^{O(k^{5^k})}$ time that decides, given a graph $G$ and $k$ pairs of vertices $(s_1, t_1), \dots (s_k, t_k)$, of the existence of $k$ internally vertex-disjoint paths $P_1, \dots, P_k$ such that each $P_i$ is a shortest $(s_i, t_i)$-path.
\end{theorem}

We also show that this problem is $W[1]$-hard parameterized by $k$, which means that we cannot hope to remove the dependency in $k$ in the exponent. However, we can extend the previous result to the case where paths are not required to be shortest paths, but of length at most $d(s_i, t_i) + C$, where $C$ is a fixed constant.  

\begin{corollary}\label{thm:SDPC}
	For any fixed integers $k$ and $C$ , there exists an algorithm running in $n^{O((Ck)^{5^k})}$ that decides, given a graph $G$ and $k$ pairs of vertices $(s_1, t_1), \dots (s_k, t_k)$, of the existence of $k$ internally vertex-disjoint paths $P_1, \dots, P_k$ such that each $P_i$ is a path of length at most $d(s_i, t_i) +C$ between $s_i$ and $t_i$. 
\end{corollary}

\section{Preliminaries}
For any integer $k$, $[k]$ denotes the set of integers between $1$ and $k$, and for any integer $j \leq k$,
$[j..k]$ denotes the set of integers between $j$ and $k$. The \textit{length} of a path $P$ correspond to the number of edges in the path, and given two vertices $x$ and $y$ in the same connected component, $d(x,y)$ denotes the minimal length of a path between $x$ and $y$.  

A graph $G$ is said to be a \textit{$k$-shortest} graph if there exists $k$ partitions of $G$ 
$(V^1_1, \dots, V^1_{l_1})$, $(V^2_1, \dots, V^2_{l_2})$, $\dots, (V^k_1, \dots, V^k_{l_k})$ 
such that $xy$ is an edge of $G$ implies that there exists $i \in [k]$ and $j \in [l_i - 1]$ 
such that $x \in V^i_{j}$ and $y \in V^i_{j+1}$. Moreover, if $x \in V^i_j$ and $y \in V^i_l$, 
for $i \in [k]$, $j,l \in [l_i]$ and $|j-l| > 1$ then $xy \not \in E(G)$. 
Intuitively, one way to obtain a $k$-shortest graph is to start from any graph, doing $k$ 
breadth-first searches from 
different vertices and removing all the edges which are not between two consecutive levels of at 
least one of the BFS. For a $k$-shortest graph $G$, we associate naturally $k$ colours to the 
edges of $G$ as follows: any edge between $x\in  V^i_{j}$ and $y \in V^i_{j+1}$ is said to be of colour 
$i$. Note that the same edge can be of different colours. Moreover, each colour $i$ defines a partial 
order on the vertices of $G$ $\leq_i$ as follows: $x \leq_i y$ if $x \in V^i_j$ and $y \in V^i_r$ 
with $j \leq r$. This naturally defines an orientation of the edges of colour $i$. Note that the same 
edge can have two different orientations for two different colours.
Let $G$ be a $k$-shortest graph and $r$ and $i$ be two indices. 
We say that a path $P_i = x_1, \dots, x_r$ is a \textit{path of colour $i$} 
if for every $j \in [r-1]$, $x_jx_{j+1}$ is an edge of colour $i$ and 
$x_j \leq_i x_{j+1}$. Note that, since whenever $u \in V^i_j$ and $v \in V^i_l$ with $|j-l|>1$ there 
is no edge between $u$ and $v$, any path of colour $i$ between $x$ and $y$
is also a shortest path in $G$. Moreover, concatenating two paths of colour $i$ also gives a path of colour $i$. 
By convention, we consider the paths of colour $i$ to be oriented from the endpoint which belongs 
to the part of $(V^i_1, \dots, V^i_{l_i})$
with the lowest index to the endpoint with the largest one. In particular, an  \textit{$(x,y)$-path of colour 
$i$} is a path of colour $i$ between $x$ and $y$ oriented from $x$ to $y$. 
For a directed path $P$ and two vertices $x$ and 
$y$ belonging to this path, $P[x,y]$ denotes the subpath of $P$ from $x$ to $y$. By convention, if $y$
is before $x$ along $P$, then $P[x,y]$ will be the empty path.
The \textit{length} of a path is its
number of edges. A \textit{path-partition} of a path $P$ is a set of internally
vertex-disjoint subpaths of $P$ such that concatenation of all the paths gives $P$. Let $Q_1, Q_2$ be two different path partitions 
of the same path $P$, the \textit{intersection} of $Q_1$ and $Q_2$ is the path partition of $P$ obtained 
as follows: If $S$ is the set of vertices which are endpoints of paths of either $Q_1$ or $Q_2$, then 
the intersection of $Q_1$ and $Q_2$ consists of all the subpaths of $P$ between vertices of $S$ which are 
consecutive along $P$. Note that the number of paths in the intersection of $Q_1$ and $Q_2$ is at most 
the sum of the number of paths in $Q_1$ and $Q_2$. Moreover, every path in the intersection is 
a subpath of some path in $Q_1$ and some path in $Q_2$. For an oriented edge $e=xy$, $x$ is called 
the \textit{tail} of $e$ and is denoted as $t(e)$ and $y$ the \textit{head}, denoted as $h(e)$.

Let $G$ be a $k$-shortest graph, $(s_1, t_1), \dots (s_l, t_l)$ a set of $l$ pairs of vertices 
and $c$ a function from $[l]$ to $[k]$, the \textit{$k$-DSP} problem defined by $G$, the $(s_i, t_i)$ and $c$
is the problem of finding a set of internally vertex-disjoint paths $P_1, \dots, P_l$ such that for any $i \in [l]$, 
$P_i$ is a path of colour $c(i)$ between $s_i$ and $t_i$. The $(s_i,t_i)$ will be referred to as \textit{requests}. The following lemma shows that we can reduce 
Eilam-Tzoreff's question to solving an instance of $k$-DSP.

\begin{lemma}\label{lem:from_DSP_to_shortest_graph}
    Let $G$ be a graph and  $(s_1, t_1), \dots, (s_k, t_k)$ be a set of $k$ pairs of vertices in $G$. 
    Let $G'$ be the $k$-shortest graph obtained from $G$ by taking for each $i$ $(V^i_1, \dots, V^i_{l_i})$
    the partition obtained by doing a breath-first-search from $s_i$, and removing the edges which 
    are not between two consecutive levels of some BFS. There exists a set of internally vertex-disjoint paths 
    $P_1, \dots, P_k$ such that for every $i \in [k]$, $P_i$ is a shortest path between $s_i$ and 
    $t_i$ if and only if the $k$-DSP problem defined by $G'$, the $(s_i, t_i)$ and the identity function 
    $c:[k] \rightarrow [k]$ has a solution. 
\end{lemma}

\begin{proof}
    The proof follows from the fact that, for every $i$ and $l$, the set $V^i_l$ corresponds to the set of 
    vertices at distance $l$ from $s_i$. Therefore, a shortest path in $G$ between $s_i$ and some vertex $
    x\in V^i_l$ is a path of colour $i$ from $s_i$ to $x$ in $G'$ and vice versa.  
\end{proof}

The main contribution of this paper is to prove the following result, which implies Theorem \ref{thm:SDP}

\begin{theorem}\label{thm:main}
    Let $G$ be a $k$-shortest graph, $(s_1, t_1), \dots, (s_l, t_l)$ a set of $l$ pairs of vertices 
    and $c$ a function from $[l]$ to $[k]$. There exists an algorithm running in time $n^{O(l^{5^k})}$
    deciding if the problem of $k$-DSP defined by $G$, the $(s_i,t_i)$ and $c$ has a solution. 
\end{theorem}

In fact, by noting that a path of length at most $d(s_i,t_i) + C$ between $s_i$ and $t_i$ uses at most $C$ edges which are not edges between consecutive levels of the BFS starting in $s_i$, trying all $\binom{n^2}{kC}$ choices for these edges allows us to reduce the problem of Corollary \ref{thm:SDPC} to a $k$-DSP problem with at most $k \cdot C$ pairs. 

An interesting case is when $k=1$. The problem then reduces 
to the problem of directed disjoint paths in acyclic digraphs by orienting all the edges of $G$ from each set
$V_i$ to $V_{i+1}$. Therefore, the algorithm of Fortune et al. \cite{FORTUNE1980111} gives a solution in $n^{O(l)}$. As noted in \cite{Bang}, we can also reduce the problem of directed-disjoint paths in acyclic digraphs 
to $1$-DSP, which implies the following theorem:

\begin{theorem}\label{th:hardness}
	The $1$-DSP problem is $W[1]$-hard parameterized by the number of requests.
\end{theorem}

\begin{proof}
	Consider an instance $D$ and $(s_1, t_1), \dots, (s_k, t_k)$ of disjoint paths in acyclic digraphs and $v_1, \dots, v_n$ a topological ordering of the vertices of $D$. Let $D'$ be the digraph obtained from $D$ by subdividing each arc $(v_i, v_j)$ $j -i - 1$ times. By doing so, every $(v_i,v_j)$-path in $D'$ now has length exactly $i-j -i$, which means that every path is a shortest path. The underlying graph $G'$ obtained by forgetting the orientation of the arcs in $D'$ is a $1$-shortest graph, as all the edges appear in the BFS starting from $v_1$, and a shortest path between $v_i$ and $v_j$ still corresponds to a $(v_i,v_j)$-path in $D$. This means that solving the $1$-DSP problem defined by $G'$ and the $(s_i, t_i)$ gives a solution to the original instance of disjoint paths in $D$, which ends the proof as disjoint paths in acyclic digraphs is $W[1]$-hard \cite{Slivkins}\footnote{The hardness proof of Slivkins, while stated for the arc-disjoint version works also for the vertex-disjoint one.}.  
\end{proof}

The rest of the paper is devoted to the proof of Theorem \ref{thm:main}. The main idea behind the proof is to reduce to a set of $O(l^{5^k})$ requests such that, roughly speaking, for each pair of requests of different colours, no pair of shortest paths solving these requests can intersect. Once we have achieved this, it means that the only potential conflicts arise for pairs of requests of the same colour. However, the edges of each colour class can be seen as an acyclic digraph, and we can adapt the algorithm of Fortune et al. to that case. The main difficulty lies in reducing to these $O(l^{5^k})$ requests. To achieve this, we need to look at a potential solution to the original $k$-DSP problem and say that, for each pair of paths of different colours in this solution, there is a way to partition each of these paths into a finite number of subpaths, such that the endpoints of each pair of subpaths now correspond to requests that can never intersect.
The next two sections are devoted to this task. In particular, it is devoted to the understanding of the structure of bi-coloured edges.

\section{Bi-coloured components}

Let $G$ be a $k$-shortest graph and $i,j$ two integers in $[k]$. Consider $G_{i,j}^+$ (resp. $G_{i,j}^-$ ), the graph induced by the edges $xy$ of $G$ of colour both $i$ and $j$ such that $x \leq_i y$ and $x \leq_j y$  (resp.  $x \leq_i y$ and $y \leq_j x$). A \textit{bi-coloured} component of colours $i,j$ is a connected component of $G_{i,j}^+$ or $G_{i,j}^-$. Note that $G_{i,j}^+$ and $G_{i,j}^-$ play identical roles, as reversing the order of the partition $(V^j_1, \dots, V^j_{l_j})$ transforms the $k$-shortest graph $G$ into a $k$-shortest graph $G'$ where every component of $G_{i,j}^-$ becomes a component of $G_{i,j}^+$ and vice-versa. Bi-coloured components will play an essential role in order to decompose a $k$-DSP problem into a set of $O(l^{5^k})$ requests such that the intersections between request of different colours behave nicely\footnote{We will define what we mean by nicely in the next section.}. The first property we need to prove is that for any path $P_i$ of colour $i$ and bi-coloured component $C$ of colour $i$ and $j$, then $P_i \cap C$ is a subpath of $P_i$. Let us first show some properties of the bi-coloured components. 

\begin{lemma}\label{lemma_diff_pos}
    Let $G$ be a $k$-shortest graph, $i,j$ two indices in $[k]$, and 
    $S$ some component of $G_{i,j}^+$.  There exists a constant $C_{S}$ such that 
    for any vertex $x \in S$, if $x \in V^i_r$, then $x \in V^j_{r + C_{S}}$.
\end{lemma}

\begin{proof}
    Let $x$ be any vertex belonging to $S$. Let $r$ and $t$ be the constants such that 
    $ x \in V^i_r$ and $x \in V^j_t$ and define $C_S = t-r$. Let $y$ be another vertex of $S$. 
    By definition of $S$, there exists a path $P$ in $G_{i,j}^+$ between $x$ and $y$. Let $s_1$ 
    be the number of edges of $P$ which are used positively for the order induced by the colour
    $i$ when going from $x$ to $y$, and $s_2$ the number of edges used negatively. By definition of 
    this order, we have that $y \in V^i_{r + s_1 - s_2}$. 

    Because the orders induced by the colours $i$ and $j$ are the same on
    $S$, we also have that $y \in V^j_{t + s_1 - s_2}$, which ends the proof. 
\end{proof}

Let us now show the following properties of paths of colour $i$.

\begin{proposition}\label{prop:shortest_path}
    Let $G$ be a $k$-shortest graph. Suppose $x \in V^i_r$ and $y \in V^i_t$ for some $i \in [k]$ 
    and $r,t \in [l_i]$ with $r > t+1$. If there exists a path in $G$ of length $r-t$ between $x$ and
    $y$, then this path is a path of colour $i$ from $y$ to $x$. 
\end{proposition}

\begin{proof}
Let $P= x_1, \dots ,  x_s$ with $x_1 = y$, $x_s = x$ and $s = r-t +1$ be a path of 
length $r-t$ between $x$ and $y$. For every $j \in [s]$, let $i_j$ be the integer such 
that $x_j \in  V^i_{i_j}$. We know that for any $j \in [2..s]$, $i_j \leq i_{j-1} + 1$ as 
$x_j$ and $x_{j-1}$ are adjacent. However, $i_1 = t$, $i_s = r$ and $s = r-t +1$. This means 
that $i_j = i_{j-1} + 1$ for every $j \in [2..s]$ and all the edges of $P$ are edges of colour $i$.
\end{proof}

\begin{proposition}\label{prop:two_path_colours}
    Let $G$ be a $k$-shortest graph, $i,j$ two indices of $[k]$ and $x$ and $y$ two vertices of $G$.
    If there exists a path $P_i$ of colour $i$ between $x$ and $y$ and a path $P_j$ of colour $j$ 
    between $x$ and $y$, then $P_j$ is also a path of colour $i$ and $P_i$ is also a path of colour 
    $j$. 
\end{proposition}

\begin{proof}
    We know that $P_i$ and $P_j$ are shortest paths between $x$ and $y$, and in particular have the same 
    length. The result follows by applying Proposition \ref{prop:shortest_path} to the paths $P_j$ and 
    $P_i$. 
\end{proof}

We are ready to prove the following lemma, which shows how paths of colour $i$ interact with 
$G_{i,j}^+$.

\begin{lemma}\label{lem:compo_inter_subpath}
    Let $G$ be a $k$-shortest graph, $i,j$ two indices in $[k]$, $P_i$ a path of colour $i$ and 
    $S$ some bi-coloured component of colours $i,j$. The intersection of $P_i$ and $S$ 
    is a subpath of $P_i$. 
\end{lemma}

\begin{proof}
    Let $S$ be a component of $G_{i,j}^+$ and suppose that $P_i$ does not intersect $S$ along a single 
    subpath. This means that we can find a subpath $P'$ of $P_i$ of colour $i$ between two vertices $x$, $y$
    of $S$ such that $P'$ uses no edge of $S$. Suppose $x$ is the first endpoint of 
    this path, $y$ the last and $l$ denote the length of $P'$. Then $x \in V^i_r$ and 
    $y \in V^i_{r+ l}$. However, by Lemma \ref{lemma_diff_pos}, we know that there exists a constant 
    $C_S$ such that, since both $x$ and $y$ belong to $S$, $x \in V^j_{r + C_S}$ and 
    $y \in V^j_{r+ l + C_S}$.  By Proposition \ref{prop:shortest_path}, this implies that
    $P'$ is also a path of colour $j$, and thus $P' \in S$.
    The case where $S$ is a component of $G_{i,j}^-$ is very similar and thus omitted. 
\end{proof}

Let $G$ be a $k$-shortest graph, $i,j$ two indices in $[k]$, $P_i$ a path of colour $i$ and $P_j$ 
a path of colour $j$. We say that $P_i$ and $P_j$ are in \textit{conflict} if there exists a bi-coloured
component $S$ of colour $i,j$ such that the intersection of $P_i$ with $S$ is a $(s_1',t_1')$-path 
and the intersection of $P_j$ with $S$ is an $(s_2', t_2')$-path 
for some $s_1', s_2', t_1', t_2' \in S$,
with the property that there exists a $(s_1',t_1')$-path $P_i'$ of colour $i$ and an 
$(s_2', t_2')$-path $P_j'$ of colour $j$ using at least one vertex outside of $s_1', s_2', t_1', t_2'$ in common. 
The component $S$ will be called a 
\textit{conflicting component} for $P_i$ and $P_j$. By convention, two paths of the same colour are 
never conflicting. 
The following lemma is the main ingredient of our Algorithm. It shows that for two paths of colour $i$ 
and $j$, there is at most one conflicting component. 

\begin{lemma}\label{lem:inter_main}
    Let $G$ be a $k$-shortest graph, $i,j$ two indices in $[k]$, $P_i$ a path of colour $i$ 
    and $P_j$ a path of colour $j$. Suppose $S$ is a conflicting component for the paths $P_i$ and 
    $P_j$, then $P_i$ and $P_j$ do not have any vertex in common outside $S$. 
\end{lemma}

\begin{proof}
    Again, we can assume that $S$ is a component of $G_{i,j}^+$ by potentially reversing the order $i$.
    Suppose that the intersection of $P_i$ with $S$ is an $(s_1',t_1')$-path and 
    the intersection of $P_j$ with $S$ is an $(s_2', t_2')$-path for some $s_1', s_2', t_1', t_2' \in S$.
    We prove the lemma by contradiction, distinguishing several cases depending on 
    which part of $P_i$ and $P_j$ (compared to $S$) the intersection lies on.

    Suppose first that the intersection of $P_i$ and $P_j$ lies after $S$ for both paths. 
    By definition of conflicting components, we know that 
    there exists a vertex $x \in S$ such that there exists a path $P_1$ of colour 
    $i$ from $x$ to $t_1'$ and a path $P_2$ of colour $j$ from $x$ to $t_2'$. Let $z$ be the first vertex 
    belonging to the intersection of $P_i \cap P_j$ after $S$. By applying Proposition \ref{prop:two_path_colours}
    to the path obtained by concatenating $P_1$ and $P_i[t_1', z]$ and the one obtained by concatenating
    $P_2$ and $P_j[t_2',z]$, we get that these two paths are both of colour $i$ and $j$. 
    In particular this implies that the edges of these paths belong to $G_{i,j}^+$ and $z \in S$,
    which is a contradiction. 

    Suppose now that the intersection of $P_i$ and $P_j$ lies before $S$ on $P_i$ and after $S$ 
    on $P_j$. By definition of conflicting components, we know that 
    there exists a vertex $x \in S \setminus \{ s_1', s_2', t_1', t_2' \}$ such that there exists a path $P_1$ of colour 
    $i$ from $s_1'$ to $x$ and a path $P_2$ of colour $j$ from $x$ to $t_2'$. 
    Because $x \not \in \setminus \{ s_1', s_2', t_1', t_2' \}$, 
    the two paths $P_1$ and $P_2$ have both length at least $1$. Moreover, 
    they are both paths of colour $i$ and $j$, and with the same orientation associated to these colours. 
    Let $z$ denote the first vertex in the intersection of $P_i$ and $P_j$ before $S$ on $P_i$ and after $S$ on $P_j$ 
    and consider the path $H_1 = P_i[z, s_1']P_1P_2$. $H_1$ is a path of colour $i$ between $z$ and 
    $t_2'$, and thus $|H_1| = |P_j[t_2',z]|$. Likewise, we can show that 
    $|P_i[z,s_1']P_1| = |P_2P_j[t_2', z]|$, which gives us a contradiction. 
    
    The other cases are symmetrical. 
\end{proof}

Let us now explain how the previous lemma will be used. Remember that our goal is to reduce an original instance of $k$-DSP with $l$ requests to one with $O(l^{5^k})$ requests such that for every pair of requests of different colours, no pair of shortest paths solving these requests can intersect. Suppose $P_i$ and $P_j$ are two paths of different colours in a solution of the original $k$-DSP which are in conflict. Let $S$ denote the conflicting component for $P_i$ and $P_j$. Because of Lemma \ref{lem:compo_inter_subpath}, we know that the intersection of $P_i$ and $P_j$ with $S$ are subpaths. For every $a \in \{i, j\}$, consider the path partition $(P^1_a, P^2_a, P^3_a)$ of $P_a$, where $P^2_a$ is the subpath of $P_a$ on $S$, $P^1_a$ the part of $P_a$ before this component, and $P^3_a$ the part after. What Lemma \ref{lem:inter_main} roughly says is that the endpoints of $P^1_i, P^1_j, P^3_i$ and $P^3_j$ correspond to requests such that no pair of shortest paths solving these requests can intersect, which is exactly what we wanted. This is not true for the requests associated to $P^2_j$ and $P^2_i$, however since they both belong to a bi-coloured component $S$, these two requests can be considered of the same colour. 

Surprisingly, the case where $P_i$ and $P_j$ are not in conflict is harder to handle. This is the goal of the next section, but let us first show sufficient conditions to guarantee the existence of a conflicting component for a pair of paths. 

\begin{lemma}\label{lem:3_conflict}
Let $G$ be a $[k]$-shortest graph and $i,j$ two different indices in $[k]$, $P_i$ a path of colour $i$ 
and $P_j$ a path of colour $j$. If $P_i$ and $P_j$ have three common vertices, then they have a 
conflicting component.
\end{lemma}

\begin{proof}
    Let $x_1$, $x_2$ and $x_3$ be three vertices in $P_i \cap P_j$. We claim that they belong to 
    the same bi-coloured component.
    Indeed, consider the subpaths of $P_i$ and $P_j$ between $x_1$ and $x_2$. By Proposition \ref{prop:two_path_colours}
    they are both paths of colour $i$ and $j$ and belong to the same component $S$. 
    Without loss of generality, suppose $S$ is a component of $G_{i,j}^+$ and $x_1 \leq_i x_2$. 
    If $x_3$ belongs to $P_i[x_1, x_2]$ or $P_j[x_1, x_2]$, we have that $x_3 \in S$, which ends the 
    proof of the claim. Assume now $x_3$ appears after $x_2$ on $P_i$, the other case being symmetrical. 
    If it appears after $x_2$ on $P_j$, then the same argument shows that $P_i[x_2,x_3]$ is also 
    a path of $S$. 
    
    Suppose now that $x_3$ appears before $x_1$ on $P_j$. In that case we have that $P_j[x_3,x_2]$ 
    and $P_i[x_2, x_3]$ are both shortest path, and thus have the same size. 
    However, this implies that $P_j[x_3,x_1]$ is strictly shorter than $P_i[x_1,x_3]$, which is a 
    contradiction. 

    Now that we know that  $x_1$, $x_2$ and $x_3$ belong to the same bi-coloured component, let us 
    show that this component is a conflicting component for $P_i$ and $P_j$.
    Indeed, since $x_1, x_2$ and $x_3$ belong to the same component, they either appear in the same order 
    on the paths $P_i$ and $P_j$ if $S$ is a component of $G_{i,j}^+$ or in reverse order
    if $S$ is a component of $G_{i,j}^-$. In both cases, the vertex in the middle is the same 
    in both paths, and this implies that $P_i$ and $P_j$ are conflicting on this component. 
\end{proof}

\section{Blind Paths}

As explained earlier, if $P_i$ and $P_j$ are two paths of the solution of some $k$-DSP problem which are in conflict, then Lemma \ref{lem:inter_main} allows us to show that the paths $P_i$ and $P_j$ can be decomposed into a finite set of requests such that each pair of requests of different colours are without any possible intersection, meaning that for any shortest path $P$ solving the first request and $P'$ solving the second request, $P \cap P' = \emptyset$. Moreover, finding these decompositions only requires to guess the conflicting component and the intersection of $P_i$ and $P_j$. Unfortunately, it is not true that any positive instance of $k$-DSP has a solution where every pair of paths of different colours are in conflict. For this purpose we need the definition of blind pair of paths. 

Let $P_i$ be some $(s_i, t_i)$-path of colour $i$ and $P_j$ some 
$(s_j, t_j)$-path of colour $j$ which are internally vertex-disjoint. 
We say that $P_i$ \textit{sees} $P_j$ if there exists an internal vertex $x$ 
of $P_i$ such that there exists a path of colour $i$ from $x$ to $t_i$ which intersects 
$P_j \setminus \{s_j, t_j\}$. 
We say that the pair $P_i$ and $P_j$ is \textit{blind} if $P_j$ does not see $P_i$ and $P_i$ does 
not see $P_j$.

Note that if $|P_i| = 2$, then $P_i$ does not see, or is not seen, by any other path $P_j$. Note also that, if $P_i$ and $P_j$ are blind, then it is possible to find a $(s_j, t_j)$-path of colour $j$ $P'_j$ and a $(s_i, t_i)$-path of colour $i$ $P'_i$ such that $P'_i \cap P'_j \not = \emptyset$. In that sense, blind paths is a weaker notion than what we could obtain from Lemma \ref{lem:inter_main} outside of the conflicting component.

The following lemma shows how to use Lemma \ref{lem:inter_main} to find blind paths from conflicting paths.

\begin{lemma}\label{lem:blind_1}
    Let $G$ be a $k$-shortest graph, $i,j$ two integers in $[k]$, $P_i$ a path of colour $i$ 
    and $P_j$ a path of colour $j$ which are internally vertex-disjoint. 
    There exists a path partition $L_i$ of $P_i$ and a 
    path partition $L_j$ of $P_j$, both of size at most $9$ with the following properties:
    \begin{itemize} 
        \item All the paths of $L_a$ are paths of colour $a$ for $a \in \{i, j\}$
        \item For any pair of paths $H_i \in L_i$ and $H_j \in L_j$, then either $H_i$ and 
        $H_j$ are paths of the same bi-coloured component of colour $i,j$ or $H_i$ does not see 
        $H_j$.
    \end{itemize}
\end{lemma}

\begin{proof}
    Suppose $P_i$ is an $(s_i, t_i)$-path and $P_j$ is an $(s_j, t_j)$ path. If $P_i$ does not see $P_j$, then 
    $L_i = \{ P_i\}$ and $L_j = \{P_j\}$ satisfy the properties of the lemma. 
    Suppose now $P_i$ sees $P_j$ and let $x_1$ denote the last vertex of $P_i$ 
    from which there exists a path $Q^1_i$ of colour $i$ to $t_i$ which uses some vertex of $P_j \setminus \{s_j, t_j\}$. 
    Because $P_i$ and $P_j$ are internally vertex disjoint, $x_1 \not = t_i$. 
    Let $x_1'$ denote the vertex just after $x_1$ on $P_i$. Note that $P_i[x_1',t_1]$ does not see $P_j$.

    Now let $x_2$ denote the last vertex of $P_i[s_i, x_1]$ from which there exists a path  
    of colour $i$ to $x_1$ which uses some vertex of $P_j \setminus \{s_j, t_j\}$. 
    Again, if this vertex does not exist, 
    then $L_i = \{P_i[s_i, x_1], (x_1,x_1'), P_i[x_1', t_i] \}$ and $L_j = \{P_j\}$ satisfy the properties of the
    lemma. Suppose from now on that $x_2$ exists and let $x_2'$ be the vertex just after $x_2$ on $P_i$.
    Since $P_i$ and $P_j$ are internally vertex-disjoint, $x_2 \not = x_1$ and thus $x_2' \in P_i[s_i, x_1]$.
    Again, note that $P_i[x_2',x_1]$ does not see $P_j$.
    
    Let $x_3$ denote the last vertex of $P_i[s_i, x_2]$ from which there exists a path $Q^3_i$ 
    of colour $i$ to $x_2$ which uses some vertex of $P_j \setminus \{s_j, t_j\}$. 
    Again, we can assume that this vertex exists or $L_i = \{ P_i[s_i, x_2], (x_2, x_2'), P_i[x_2',x_1], (x_1, x_1'), P_i[x_1', t_i] \} $
    and $L_j = \{P_j\}$ satisfy the properties of the lemma. Let $x_3' \in P_i[s_i, x_2]$ denote the vertex just after $x_3$ on $P_i$. 
    Again, note that  $P_i[x_3',x_2]$ does not see $P_j$. 

    Note that for any internal vertex $x \in Q^1_i$ and $y \in Q^2_i$, $y <_i x$. This implies that the intersection 
    of $ Q^1_i$ and $Q^2_i$ is equal to $x_1$, and the same argument applies for $Q^3_i \cap Q^2_i$ and $Q^3_i \cap Q^1_i$.   
    This means that the paths $P_j$ and $P'_i = P_i[s_i, x_3]Q^3_iQ^2_iQ^1_i$ intersect on at least 
    3 vertices and thus are conflicting by Lemma \ref{lem:3_conflict}. Let $S$ denote the conflicting component of $P_i'$ and $P_j$. 
        
    Suppose first that none of the $s_i, t_i, s_j, t_j$ belong to $S$ and denote by $e_i$ the last edge of $P'_i$ without both endpoints in $S$, $e_j$ the last edge of $P_j$ before $S$, $h_i$ the first edge of $P'_i$ after $S$ and $h_j$ the first edge of $P_j$ after $S$. 

    \begin{claim}
        All the pairs of paths among $P'_i[s_i, t(e_i)]$,$e_i$, $P'_i[h(e_i), t(h_i)]$, $h_i$, $P'_i[h(h_i), t_i]$, $P_j[s_j, t(e_j)]$, $e_j$, $P_j[h(e_j), t(h_j)]$, $h_j$ and $P_j[h(h_j), t_j]$ are blind, except from $P_j[h(e_j), t(h_j)]$ and $P'_i[h(e_i), t(h_i)]$ which belong to the same bi-coloured component. 
    \end{claim}

    \begin{proof}
        Since $e_i$ and $h_i$ are not edges of 
        $S$ and there exists a path of colour $i$ in this component from $h(e_i)$ to $t(h_i)$, 
        then by Lemma \ref{lem:compo_inter_subpath} no path of colour $i$ from $s_i$ to $t(e_i)$ or from $h(h_i)$ to 
        $t_i$ can use any vertex of $S$. However, any path of colour $j$ from $h(e_j)$ to $t(h_j)$ is a path of $S$, 
        so it cannot intersect any path of colour $i$ from $s_i$ to $t(e_i)$ or from $h(h_i)$ to $t_i$. This means that $(P'_i[s_i, t(e_i)],P_j[h(e_j), t(h_j)])$ and $(P'_i[h(h_i), t_i],P_j[h(e_j), t(h_j)])$ are blind pairs. By reversing the role of $i$ and $j$, it also means that $(P_j[s_j, t(e_j)],P_i'[h(e_i), t(h_i)])$ and $(P_j[h(h_j), t_j],P_i'[h(e_i), t(h_i)])$ are blind pairs.
    
        By the definition of conflicting and Lemma \ref{lem:inter_main}, we can show that 
        no path of colour $i$ from $s_i$ to $t(e_i)$ or from $h(h_i)$ to $t_i$ can 
        intersect a path of colour $j$ from $s_j$ to $t(e_j)$ or $h(h_j)$ to $t_j$. Indeed, suppose for example that 
        there exists a path $H_i$ of colour $i$ from $s_i$ to $t(e_i)$ that intersects a path $H_j$ 
        of colour $j$ from $s_j$ to $t(e_j)$. In that case the paths $H_iP_i'[t(e_i), h(h_i)]$ and 
        $H_jP_j[t(e_j), h(h_j)]$ contradict Lemma \ref{lem:inter_main} as $S$ is a conflicting component 
        for these two paths, but they also intersect outside of $S$. The other cases are symmetrical and thus all pairs of paths among $P'_i[s_i, t(e_i)]$, $P'_i[h(h_i), t_i]$, 
        $P_j[s_j, t(e_j)]$ and $P_j[h(h_j), t_j]$ are blind. 
        
        This ends the proof of the claim as the other pairs contain an edge and are blind by definition and $P_j[h(e_j), t(h_j)]$ and $P'_i[h(e_i), t(h_i)]$ are paths of $S$.
    \end{proof}

    Let $L_j = \{P_j[s_j, t(e_j)], e_j, P_j[h(e_j), t(h_j)], h_j, P_j[h(h_j), t_j] \}$. 
    Suppose first that $t(e_i)$ appears after $x_3$ on $P'_i$. It means that $P_i[s_i, x_3]$ is a subpath
    of $P'_i[s_i, t(e_i)]$, and in particular $P_i[s_i, x_3]$ does not see $P_j$. Setting $L_i = \{P_i[s_i, x_3],(x_3, x_3'),P_i[x_3', x_2], (x_2, x_2'), P_i[x_2', x_1], (x_1,x_1'), P_i[x_1', t_i], \}$, we then have that no path of $L_i$ sees $P_j$ and thus any path of $L_j$. 

    Suppose now that $t(e_i)$ appears before $x_3$ on $P'_i$. Note that $t(h_i)$ has to appear after $x_3$ or there is 
    no path from $x_3$ to $t_i$ intersecting $P_j$, which contradicts the choice of $x_3$.
    In that case, setting $L_i = \{P_i[s_i, t(e_i)], e_i, P_i[h(e_i), x_3], (x_3, x_3'), P_i[x_3, x_2], (x_2, x_2')
    P_i[x_2', x_1], (x_1, x_1'), P_i[x_1, t_i] \}$, we also have 
    that the only path of $L_i$ that sees a path of $L_j$ is $P_i[h(e_i), x_3]$. 
    Moreover, it can only see $P_j[h(e_j), t(h_j)]$, but these paths belong to the same bi-coloured component $S$.

    The cases where some of the $s_i, t_i, s_j, t_j$ belong to $S$ are treated exactly the same, except that some of the $e_i, e_j, h_j, h_i$ might not exist, which means we have fewer paths to consider.
\end{proof}

By applying the previous lemma several times, we obtain the following: 

\begin{lemma}\label{lem:blind_main}
    There exists a constant $C$ such that if $G$ is a $k$-shortest graph, $i,j$ two integers in $[k]$,
    $P_i$ a path of colour $i$ and $P_j$ a path of colour $j$ which are internally vertex-disjoint,
    then there exists a path partition $L_i$ of $P_i$ and a 
    path partition $L_j$ of $P_j$, both of size at most $C$ with the following properties:
    \begin{itemize}
        \item Each $L_a$ consists of at most $C$ paths of colour $a$. 
        \item For any pair of path $H_i \in L_i$ and $H_j \in L_j$ which are not blind, 
        then $H_i$ and $H_j$ are paths of the same bi-coloured component. 
    \end{itemize}
\end{lemma}

\begin{proof}
    Let $Q_i$, $Q_j$ be the path partitions obtained by applying Lemma \ref{lem:blind_1} to $P_i$ and $P_j$.
    We know that for any pair of paths $H_i \in Q_i$ and $H_j \in Q_j$, then either $H_i$ and $H_j$ are 
    paths of the same bi-coloured component of colours $i,j$, or $H_i$ does not see $H_j$. 
    
    Now as long as there exists a path in $H_i \in Q_i$ such that there exists 
    some path  $H_j \in Q_j$,
    such that $H_j$ sees $H_i$ and is not a path of the same bi-coloured component as $H_i$, 
    we do the following. 
    Let $H_{j,1}, \dots, H_{j,r}$ denote all the paths of $Q_j$ which see $H_i$ and do not belong 
    to the same bi-coloured component. For any $a \in [r]$, let $Q_{j,a}$ and $Q_{i,a}$ 
    denote the set of path partitions obtained by applying Lemma \ref{lem:blind_1} to $H_{j,a}$ and $H_i$. 
    Let $Q'_i$ be the intersection of all the partitions $Q_{i,a}$ of $P_i$. Because every path of $Q'_i$
    is a subpath of some $Q_{i,a}$ for any $a \in [r]$, it means that this path is not seen by any path in $Q_{j,a}$ which is not a path of the same bi-coloured component. 
    Let us update $Q_i$ by replacing 
    $H_i$ by $Q'_i$ and update $Q_j$ by replacing each of the $H_{j,a}$ by $Q_{j,a}$. 
    By doing that, the number of paths in $Q_i$ which is seen by some path $H_j \in Q_j$ which is not a path of the same bi-coloured component decreases strictly as none of the paths of $Q'_i$ satisfy these properties. At each step, we multiply the number of paths in $Q_j$ by at most 9 and the number of paths in $Q_i$ by at most $9|Q_j|$. However, we only have to do this $9$ steps as initially the sets $Q_i$ and $Q_j$ have size at most $9$. Finally, this means that after 9 steps, $|Q_j| \leq 9^9$ and $|Q_i|\leq  9(9|Q_j|)^{9} \leq 9^{91}$. This ends the proof for $C = 9^{91}$
\end{proof}

\section{Proof of the main theorem}

We are now ready to describe and prove our algorithm. First we will show, using Lemma \ref{lem:blind_1} and induction, that any solution $P_1, \dots, P_l$ of some $k$-DSP can be decomposed into some path partitions $L_1, \dots, L_l$ where each $L_i$ has size at most $C(k,l)$ for some function $C$ and any pair of paths in the union of the $L_i$ of different colours is blind. 
The algorithm then consists of guessing the endpoints and colours of the partitions $L_1, \dots, L_l$ (there is at most $n^{O(C(k,l))}$ possible choices) and then solve the $k$-DSP defined using the fact that all the pairs of paths of different colours are blind. Roughly speaking, each colour class is solved almost independently using Fortune et al. Algorithm.

\subsection{Proof of the blind case}

Let us first show the existence of the algorithm in the blind case. Note that we also have some list of forbidden components for each path in the partitions $L_1, \dots, L_l$. This is due to some technicalities in the proof of the existence of such decomposition. 

\begin{lemma}\label{lem:blind_case}
    Let $G$ be a $k$-shortest graph, $(s_1, t_1), \dots, (s_l, t_l)$ a set of pairs and $c$ a function 
    from $[l]$ into $[k]$. Moreover, suppose that for every $i$, there 
    is a list $F_i$ of bi-coloured components where one of the colours being $c(i)$.
    There exists an algorithm running in time $
    n^{O(l)}$ that either returns a  solution $P_1, \dots, P_l$ to the $k$-DSP defined by $G$, 
    the $(s_i, t_i)$ and $c$ or shows that no solution is such that each $P_i$ does not use any 
    vertex of any component in $F_i$ 
    and moreover, for any indices $i$ and $j$, either $P_j$ and $P_i$ are blind or $P_i$ is a path 
    of some component of $F_j$ or $P_j$ is a path of some component of $F_i$.   
\end{lemma}

To prove this lemma, we will build an auxiliary digraph $D$ such that a solution satisfying the 
properties of the lemma exists if and only if there exists a directed path in $D$ 
between two specified vertices. 

First note that, by potentially replacing some vertices with an independent set with the same 
neighbourhood, we can assume that all the $s_i$ and $t_i$ are disjoint. 

The vertices of $D$ will correspond to $l$-tuples $(x_1, \dots, x_l)$ of vertices of $G$. 
Intuitively, we are trying to build the paths $P_i$ starting 
from $s_i$, and $x_i$ is the last vertex of a prefix of $P_i$ we are considering. For any pair of vertices 
$(x_1, \dots, x_l)$ and $(y_1, \dots, y_l)$, $D$ contains the arc from $(x_1, \dots, x_l)$ to
$(y_1, \dots, y_l)$ if the following are satisfied:
\begin{itemize}
    \item There exists $i \in [l]$ such that $x_j = y_j$ for all $j \in [l]$, $j \not = i$.
    \item $x_iy_i$ is an edge of colour $c(i)$ such that there exists a path of colour $c(i)$ 
    from $y_i$ to $t_i$ avoiding the components in $F_i$. 
    \item For all $j \in [l]$ different from $i$, $y_i \not = x_j$ and either there is no path of colour $c(j)$ from 
    $x_j$ to $t_j$ that uses the vertex $x_i$, or $x_i$ is a vertex of a component of $F_j$.
\end{itemize}

Let $S = (s_1, \dots, s_l)$ and $T = (t_1, \dots t_l)$.
The next two claims finishes the proof of Lemma \ref{lem:blind_case}.

\begin{claim}
    If there exists a solution $P_1, \dots, P_l$ to the $k$-DSP defined by $G$, 
    $(s_i, t_i)$ and $c$ such that each $P_i$ does not use a vertex of any component in $F_i$ 
    and moreover, for any indices $i$ and $j$, either $P_j$ and $P_i$ are blind, $P_i$ is a path 
    of some component of $F_j$ or $P_j$ is a path of some component of $F_i$, then there is 
    a path in $D$ from $S$ to $T$.
\end{claim}

\begin{proof}
    Let $P_1, \dots P_l$ denote such a solution in $G$. Let $X$ be the set of vertices of $D$
    corresponding to $l$-tuples obtained by taking one vertex per path $P_i$. 
    We can define a natural order on $X$ by considering for 
    each $P_i$ the order induced by the path and taking the lexicographic order. Note that $T$ is 
    the maximal element of $X$. 
    
    Consider now the largest element $A=(x_1, \dots, x_l)$ of $X$ which is reachable in $D$ from $S$ 
    and suppose, in order to reach a contradiction, that this element is not $T$. Consider some colour 
    $c_1$ and $I$ the set of indices of $i$ of $[l]$ such that for $c(i) = c_1$ and $x_i \not = t_i$. 
    Because the edges of colour $c_1$ induce an acyclic digraph, there exists an index $i \in I$ such 
    that for every $j \in I$ with $j \not = i$, there is no path of colour $c_1$ from $x_j$ to 
    $x_i$. 
    Now for any $j \in [l]$ such that $c(j) \not = c_1$, then either the path $P_i$ and $P_j$ 
    are blind, in which case there is no 
    path of colour $c(j)$ from $x_j$ to $t_j$ that uses $x_i$, or $P_i$ (and thus $x_i$) is in some component of $F_j$, or $P_j$ is a path of some component of $F_i$. Note that in the 
    last case, any path from $x_j$ to $t_j$ is a path of some component of $F_i$, but $x_i$ 
    cannot be a vertex of this component, and thus no such path can use $x_i$.
    Therefore, if we note $x_i'$ the 
    vertex just after $x_i$ on $P_i$ and $A'$ the vertex of $D$ obtained from $A$ by only changing 
    $x_i$ into $x_i'$, then there exists an arc from $A$ to $A'$. However, this means that $A'$ 
    is reachable from $S$ in $D$, which contradicts the maximality of $A$. 
\end{proof}

And the opposite direction. 

\begin{claim}
    If there exists a path from $S$ to $T$ in $D$, then 
    there exists a solution $P_1, \dots, P_l$ to the DSP defined by $G'$, 
    the $(s_i, t_i)$ and $c$ such that $P_i$ does not use any vertex of any component in $F_i$.
\end{claim}

\begin{proof}
    Suppose there exists a path $P = X_1, \dots, X_r$ from $S$ to $T$ in $D$. For every $j \in [r]$, 
    note $X_j = (x^j_1, \dots, x^j_l)$. For every $i$ and $j$, consider the graph $P^j_i$ induced 
    by the vertices $x^t_i$, for $t \leq j$. By definition of $D$, $P^j_i$ is a path of colour $c(i)$ 
    from $s_i$ to $x^j_i$ avoiding the components of $F_i$. 
    We will prove by induction on $j$, that the paths $P^j_1, \dots, P^j_l$ are such that 
    \begin{itemize}
        \item All the $P^j_i$ are internally disjoint.
        \item For any $i$ and $r$, there is no path of colour $c(i)$ from $x^j_i$ to $t_j$
        avoiding the components in $F_i$ that uses any vertex of $P^j_r$ outside of possibly $x^j_r$.
    \end{itemize}
    Since the path starts at $S$, all the properties are satisfied when $j = 1$. Suppose now that this 
    is true for some $j \in [r-1]$ and let us show that the properties hold for $j+1$. 
    By definition of the arcs of $D$, there exists an index $i$ such that $x^j_ix^{j+1}_i$ is an edge 
    of colour $c(i)$ and for every other index $s$, $x^j_s = x^{j+1}_s$. This means that $P^{j+1}_i$ is the concatenation
    of $P^j_i$ with $x^{j+1}_i$ and all the $P^{j+1}_s$ are equal to $P^{j}_s$ for $s \not = i$. 
    Moreover, by definition of $D$, we know that for any $s \not = i$, $x^{j+1}_i$ is disjoint from all 
    the $x^j_s$ and by induction hypothesis $x^{j+1}_i$ does not belong to any of the $P^j_s$.
    This implies that the $P^j_s$ for $s \in [l]$, are disjoint.
    
    Any path of 
    colour $c(i)$ from $x^{j+1}_i$ to $t_i$ avoiding the components of $F_i$ is a subpath of a 
    path of colour $c(i)$ from $x^{j}_i$ to $t_i$ avoiding the components of $F_i$. This means 
    that no such path can use any vertex of $P^j_s = P^{j+1}_s$ outside of possibly $x^j_s$ for all $s \in [l]$
    different from $i$.

    Finally, for $s \in [l]$ different from $i$ we know that no path of colour $c(s)$ from 
    $x^j_r = x^{j+1}_r$ to $t_s$ avoiding the components in $F_s$ can use any vertex of 
    $P^j_i$ outside of $x^i_s$ by induction hypothesis. Moreover, these paths can also not use 
    $x^{j+1}_i$ by definition of the arcs of $D$, which ends our induction.

    This means that each $P^r_i$ is a path of colour $c(i)$ avoiding the components in $F_i$ from 
    $s_i$ to $t_i$, and all these paths are disjoints, which ends the proof. 
\end{proof}

Therefore, the problem reduces to deciding the existence of a path in $D$. 
As $|D| = n^l$, this can be done in $n^{O(l)}$.

\subsection{Reducing to the blind case}

The next lemma shows how to reduce to the blind case with some forbidden lists. 

\begin{lemma}\label{lem:decomp_sol}
    Let $G$ be a $k$-shortest graph, 
     $(s_1, t_1), \dots (s_l, t_l)$ a set 
    of pairs and $c$ a function from $[l]$ to $[k]$. 
    Let $P_1, \dots, P_l$ be a solution to the $k$-DSP defined by $G$, the $(s_i,t_i)$ and $c$.
    There exists a constant $C(k,l)$ depending only on $k$ and $l$, a set of path partitions $L_1, \dots, L_l$, a function $a$ that associates to each path of the $L_i$ a colour in $[k]$  
    and a function $b$ that associates to each path of the $L_i$ a set of bi-coloured components 
    with the following properties:
    \begin{itemize}
        \item For every $i \in [l]$, $L_i$ is a path partition of $P_i$ of at most $C(k,l)$ paths.
        \item If $P$ is a path of some $L_i$ with $a(P) = c_1$, then $P$ is a path of colour $c_1$ and 
        $b(P)$ consists of a set of at most $C(k,l)$ 
        bi-coloured components where one of the colours is $c_1$ and such that $P$ does not use any vertex
        in these components.
        \item For any pair of paths
        $H_i \in L_i$ and $H_j \in L_j$ such that $a(H_i) \not = a(H_j)$, 
        then either $H_i$ and $H_j$ are blind, $H_j$ is a path contained in one component of $b(H_i)$
        or $H_i$ is a path contained in one component of $b(H_j)$.  
    \end{itemize}

\end{lemma}

Note that if $P_i$ is a path of colour $j$, then any path partition of $P_i$ consists of paths of colour $j$.
The function $a$ is there to reassign the colour of some paths belonging to bi-coloured components in 
order to achieve the last property of the lemma.

\begin{proof} 
    Let $C$ be the constant from Lemma \ref{lem:blind_main}. We will prove by induction on 
    $k$ that the lemma is true with $C(k,l) \leq  (7Cl)^{5^k}$. When $k =1$, there is only one colour 
    and setting $L_i = \{P_i\}$ for all $i$ satisfies the properties, and thus $C(1,l) \leq l$.

    Suppose now that $k >1$, and let $I_1$ denote the set of indices $i \in [l]$ such that $c(i) = 1$. Note that we can assume $I_1$ to be non empty, or the induction step is trivial. For every pair of indices $i,j \in [l]$, let $Q_{i,j}, Q_{j,i}$ denote the path partitions of 
    $P_i$, $P_j$ obtained by applying Lemma \ref{lem:blind_main} to this pair and 
    for every $i \in [l]$, let $Q_i$ denote the intersection of all the $Q_{i,j}$. Note that, since 
    every $Q_{i,j}$ has size $C$, this implies that the $Q_i$ have size at most $C l$. 
    For any $i \in I_1$, let $Q'_i$ denote the set of paths of $Q_i$, for which there exists some other path among some $Q_j$ with $j \in [k] \setminus I_1$ such that the pair is not blind. 
    By Lemma \ref{lem:blind_main}, all these paths belong to some bi-coloured component 
    of colour $1$ and some other colour $t$. Let $B_i$ denote the set of all these bi-coloured components and let $R_i = Q_i \setminus Q'_i$. Let $R = \bigcup_{ i \in I_1} R_i$, $B = \bigcup_{ i \in I_1} B_i$ and note that $|B| \leq Cl^2$ and $|R| \leq Cl^2$.
    
    For any path $R_t \in R$, we know that the intersections of $R_t$  
    with any component $C_j \in B$ is a subpath by Lemma \ref{lem:compo_inter_subpath}. 
    Let $e_j$ be the last edge of $R_t$ before $C_j$ and $h_j$ the first edge after. Let $a_t$ 
    and $b_t$ denote the first and last vertex of $R_t$, and consider 
    $R_{t,j} = \{R_t[a_t, t(e_j)], e_j, R_t[h(e_j), t(h_j)],$ $ h_j, R_t[h(h_j), b_t] \}$
    a path partition of $R_t$. Note that, except the two edges $e_j$ and $h_j$, each 
    path of $R_{t,j}$ is either disjoint from $C_j$ or a path of this component.
    Let $L(R_t)$ denote the path 
    partition of $R_t$ obtained by taking the intersection of all the $R_{t,j}$. We know that, 
    since $|B| \leq Cl^2$, $L(R_t) \leq 5Cl^2$. Moreover, we know that for any path $P'$ of $L(R_t)$, 
    and any $C_j \in B$, there is a path $r_{t,j} \in R_{t,j}$ such that $P'$ is a subpath of $r_{t,j}$.
    In particular it means that $P'$ is either an edge, disjoint from $C_j$, or a path of $C_j$. 
    Let $R^2_t$ be the set of paths of $L(R_t)$ which belong 
    to one of the component of $B$, and $R^1_t = L(R_t) \setminus R^2_t$. Note that every path in $R^1_t$ is either 
    an edge or a path disjoint from all the components of $B$.
 
    Let $H_1$ denote the set of paths in all the $Q'_i$ for $i \in I_1$ and all the $R^2_t$ for $R_t \in R$. Note that for every path $H' \in H_1$, $H'$ 
    is path of a bi-coloured component of $B$. Denote by $c'(H')$ the colour of this component which is not $1$. 
    We will now consider $H'$ as a path of colour $c'(H')$ 
    (possibly reversing the endpoints if the component is a component of $G_{1,c'(H')}^-$).
    Let $G_1$ be the $(k-1)$-shortest graph obtained from $G$ be removing the partition associated to 
    colour $1$ and removing all the edges which are edges of colour $1$ only. 
    Consider now the $(k-1)$-DSP problem defined on $G_1$ by all the endpoints of the paths in $Q_i$ for $i  \in [l]$, $i \not \in I_1$, 
    considered as path of colour $c(i)$, and all the paths in $H' \in H_1$ considered as path of 
    colour $c'(H')$. Note that the set of paths in $Q_i$ and $H_1$ is a solution to this problem and moreover, 
    there is at most $7(Cl^2)^2$ requests, as each $Q_i$ is smaller than $Cl$ and $H_1$ is smaller than $6(Cl^2)^2$. 
    
    By induction hypothesis, there exists a path partition of all the paths of the $Q_i$ for $i  \in [l]$, $i \not \in I_1$
    and $H_1$ as well as two functions $a'$ and $b'$ defined on these paths such that each of these 
    path partitions consists of at most $C(k-1, 7(Cl^2)^2)$ paths, and $b'$ associates to each path 
    at most $C(k-1, 7(Cl^2)^2)$ bi-coloured components. 
    Let us define the path partitions $L_i$, as well as $a$ and $b$ as follows: 
    For every $i \in I_1$, 
    $L_i$ is the union of all the paths in $R^1_t$ for some $R_t \in R_i$ as well as all the paths in the path partitions of the paths in $R^2_t$ and $Q'_i$ obtained by applying induction. For every path $P$ of some $R^1_t$, 
    let $a(P) = 1$ and $b(P) = B$. For all the other paths of $L_i$, $a$ and $b$ correspond to the value 
    of $a'$ and $b'$ on these paths. 
    Likewise, for every $i  \in [l]$ such that $i \not \in I_1$, $L_i$ consists of the union of the path partitions for the paths 
    in $Q_i$ obtained by applying induction, and the function $a$ and $b$ correspond to the $a'$ and $b'$ on these paths. 

    Let us now show that the $L_i$ and functions $a$ and $b$ satisfy the required properties. 
    First, it is clear that the $L_i$ thus defined are path partitions, as they are obtained by 
    replacing paths of some path partitions by their own path partition. Moreover, $|L_i| \leq 
    7(Cl^2)^2 \cdot C(k-1, 7(Cl^2)^2) \leq 7(Cl^2)^2 \cdot (7^2(Cl^2)^2)^{5^{k-1}} \leq  (7Cl)^{5^k} $. 
    Likewise, for any paths $P$ in these partitions, $b(P)$ is smaller than 
    $\max\{C(k-1, Cl^2), |B|\} \leq (7Cl)^{5^k} $. 
    Now suppose $H_i$ is a path of $L_i$ and $H_j$ is a path of $L_j$ such that $a(H_i) \not = a(H_j)$.
    If none of these paths belong to some $R^1_t$ for $R_t \in R$, then the last property of the lemma is satisfied for 
    $H_i$ and $H_j$ by induction and because $a$ and $b$ correspond to $a'$ and $b'$ on these paths. 
    Suppose now that one of the paths, say $H_i$ belongs to $R^1_t$ for some $R_t \in R$. 
    Because $R^1_t$ is a subpath of an element of $R$, 
    it means that if $H_j$ is not a subpath of some path of $Q'_s$ for $s \in I_1$, then by definition of $R$, $H_i$ and $H_j$ are blind.
    However, if $H_j$ is a subpath of some path of $Q_s'$, then $H_j$ is a path belonging to some component of $B$. However, $b(H_i) = B$, which ends the proof.
\end{proof}

Finally, we can prove our main result.

\begin{proof}[Proof of Theorem \ref{thm:main}]
    Suppose there exists a solution $P_1, \dots, P_l$ to the $k$-DSP problem defined by
    $G$, the $(s_i,t_i)$ and $c$. Let $L_1, \dots L_l$, $a$ and $b$ be the path partitions 
    and functions obtained by applying Lemma \ref{lem:decomp_sol} to $P_1, \dots, P_l$. For every $i \in [l]$, let $P_{i,1}, \dots, P_{i,l_i}$ denote the paths of $L_i$ and $(s^i_1, t^i_1), \dots (s^i_{l_i}, t^i_{l_i}) $ the endpoints of these paths. Remember that by Lemma \ref{lem:decomp_sol}, $|L_i| \leq C(k,l)$ for all $i \in [l]$. 

    Suppose we guess all the $(s^i_j, t^i_j)$, as well as the functions $a$ and $b$ for each of 
    the $P_{i,j}$, and consider the $k$-DSP problem defined by all the remaining 
    pairs $(s^i_j, t^i_j)$, then the set of paths $P_{i,j}$ is a solution to this problem such that, for any pair of 
    paths $P_{i,j}$, $P_{i',j'}$ such that $a(P_{i,j}) \not = a(P_{i',j'})$, either $P_{i,j}$ and 
    $P_{i,j}$ are blind, $P_{i,j}$ is a path contained in one component of $b(P_{i',j'})$ or 
    $P_{i',j'}$ is a path contained in one component of $b(P_{i,j})$. 
    This means that we can apply the algorithm of Lemma \ref{lem:blind_case} to find a solution of 
    the $k$-DSP defined by $(s^i_j, t^i_j)$ in $n^{O(C(k,l))}$. 
    By concatenating for each $i$ all the paths of this solution corresponding to the paths of $L_i$, we obtain a solution to the initial $k$-DSP problem.  
    
    As there is at most $n^{O(C(k,l))}$ choices for the $(s^i_j, t^i_j)$, $a(P_{i,j})$ and $b(P_{i,j})$ this gives an algorithm running in time $n^{O(C(k,l))}$, which ends the proof.
\end{proof}

\section*{Acknowledgements}
The author wishes to thank Frédéric Havet and Saket Saurabh for their useful comments on the manuscript.


\begin{thebibliography}{10}
	
	
	\bibitem{Bang}
	J.~Bang{-}Jensen, T.~Bellitto, W.~Lochet, and A.~Yeo.
	\newblock The directed 2-linkage problem with length constraints.
	\newblock {\em arxiv.org/abs/1907.00817}, 2019.
	
	\bibitem{berczi2017directed}
	K.~B{\'e}rczi and Y.~Kobayashi.
	\newblock The directed disjoint shortest paths problem.
	\newblock In {\em 25th Annual European Symposium on Algorithms (ESA 2017)}.
	Schloss Dagstuhl-Leibniz-Zentrum fuer Informatik, 2017.
	
	\bibitem{Bjorklund}
	A.~Björklund and T.~Husfeldt.
	\newblock Shortest two disjoint paths in polynomial time.
	\newblock {\em SIAM Journal on Computing}, 48(6):1698--1710, 2019.
	
	\bibitem{CHUDNOVSKY2015582}
	M.~Chudnovsky, A.~Scott, and P.D. Seymour.
	\newblock Disjoint paths in tournaments.
	\newblock {\em Advances in Mathematics}, 270:582 -- 597, 2015.
	
	\bibitem{EILAMTZOREFF1998113}
	T.~Eilam-Tzoreff.
	\newblock The disjoint shortest paths problem.
	\newblock {\em Discrete Applied Mathematics}, 85(2):113 -- 138, 1998.
	
	\bibitem{FORTUNE1980111}
	S.~Fortune, J.~Hopcroft, and J.~Wyllie.
	\newblock The directed subgraph homeomorphism problem.
	\newblock {\em Theoretical Computer Science}, 10(2):111 -- 121, 1980.
	
	\bibitem{karp1975}
	R.~M. Karp.
	\newblock On the computational complexity of combinatorial problems.
	\newblock {\em Networks}, 5(1):45--68, 1975.
	
	\bibitem{Seymour95}
	N.~Robertson and P.D. Seymour.
	\newblock Graph minors .xiii. the disjoint paths problem.
	\newblock {\em Journal of Combinatorial Theory, Series B}, 63(1):65 -- 110,
	1995.
	
	\bibitem{schrijver1990paths}
	A.~Schrijver, L.~Lovasz, B.~Korte, H.~Promel, and RL~Graham.
	\newblock Paths, flows, and VLSI-layout.
	\newblock 1990.
	
	\bibitem{schrijver1994finding}
	Alexander Schrijver.
	\newblock Finding k disjoint paths in a directed planar graph.
	\newblock {\em SIAM Journal on Computing}, 23(4):780--788, 1994.
	
	\bibitem{Slivkins}
	A.~Slivkins.
	\newblock Parameterized tractability of edge-disjoint paths on directed acyclic
	graphs.
	\newblock {\em SIAM Journal on Discrete Mathematics}, 24(1):146--157, 2010.
	
\end{thebibliography}
\end{document}